\documentclass[12pt]{amsart}
\usepackage[utf8]{inputenc}
\usepackage[margin=1in]{geometry}

\usepackage[
    backend=biber,
    style=ieee,
    citestyle=numeric-comp,
    sorting=nyt,
    dashed=false
  ]{biblatex}
\usepackage{hyperref}
\usepackage[hyphenbreaks]{breakurl}
\usepackage{xurl}
\hypersetup{
    colorlinks=true,
    linkcolor=blue,
    citecolor=red,
    breaklinks=true
}

\usepackage{graphicx}
\usepackage{amsmath}
\usepackage{amssymb}
\usepackage{appendix}
\usepackage{amsthm}
\usepackage{booktabs}
\usepackage{graphicx}
\usepackage{subcaption}
\usepackage{comment}
\usepackage{enumerate}
\usepackage{esint}

\usepackage[dvipsnames]{xcolor}

\newtheorem{theorem}{Theorem}[section]
\newtheorem{corollary}{Corollary}[theorem]
\newtheorem{prop}[theorem]{Proposition}
\newtheorem{lemma}[theorem]{Lemma}

\theoremstyle{definition}
\newtheorem{definition}[theorem]{Definition}
\theoremstyle{remark}
\newtheorem{remark}[theorem]{Remark}

\usepackage{amsmath}
\usepackage{amssymb}
\usepackage{mathrsfs}

\newcommand{\R}{{\mathbb R}}

\newcommand{\F}{{\mathcal{F}}}

\newcommand{\cS}{{\mathcal{S}}}

\newcommand{\sS}{{\mathscr{S}}}
\newcommand{\fS}{{\mathfrak{S}}}
\newcommand{\tr}{{\text{tr}}}

\usepackage{todonotes}


\title[Quantum Decoupling]{Decoupling for Schatten class operators in the setting of Quantum Harmonic Analysis}

\author{Helge J. Samuelsen}
\address{Department of Mathematical Sciences,
         Norwegian University of Science and Technology,
         Trondheim, Norway}

\email{helge.j.samuelsen@ntnu.no}
\date{\today}

\begin{document}

\begin{abstract}
We introduce the notion of decoupling for operators, and prove an equivalence between classical $\ell^qL^p$ decoupling for functions and $\ell^q\cS^p$ decoupling for operators on bounded sets in $\R^{2d}$. We also show that the equivalence only depends on the bounded set $\Omega$, and not the values of $p,q$ nor the partition of $\Omega$. The proof relies on a quantum version of Wiener's division lemma.
\end{abstract}
\maketitle

\section{Introduction}
In this paper we consider decoupling for continuous linear operators from the Schwartz class $\sS(\R^d)$ into the tempered distributions $\sS'(\R^d)$. Classical decoupling has in recent years had an enormous impact in many different areas of analysis. First introduced by Wolff for the cone in \cite{Wolff_00}, it later grew to have applications in harmonic analysis, number theory and the theory of partial differential equations. It was used by Bourgain, Demeter and Guth to prove the Vinogradov's mean value conjecture \cite{Bourgain_Demeter_Guth_16}. This was done by achieving an upper bound on the growth of  the decoupling constant for the moment curve $\Gamma:[0,1]\to \R^n$ given by $\Gamma(t)=(t,t^2,\ldots,t^{n})$. A different proof for decoupling of the moment curve can also be found in \cite{Gao_Li_Young_Zorin-Kranich_21}.

In what follows, we will extend the notion of decoupling to the setting of operators through the techniques of quantum harmonic analysis as introduced by Werner in \cite{Werner}. Several classical results from harmonic analysis have already been extended in this fashion. For instance, an operator analogue of Wiener's Tauberian theorem was considered in \cite{Luef_Skrettingland_21,Fulsche_Luef_Werner_24}.

We start by introducing the \textit{classical} decoupling constant in order to set some notation. Unlike the classical literature on decoupling, we have opted to restrict ourselves to the case $n=2d$, and replace the classical Fourier transform by the symplectic Fourier transform denoted $\F_\sigma$. This reflects the importance of the symplectic structure of the phase space $\R^{d}\times \widehat{\R^{d}}\cong \R^{2d}$.
\begin{definition}
Let $\Omega\subset \R^{2d}$, and $\mathcal{P}_{\Omega}$ be a partition of $\Omega$. We define the classical decoupling constant $\mathcal{D}^{\mathscr{C}}_{p,q}(\mathcal{P}_\Omega)$ to be the smallest constant for which the inequality
\[
\left\|\sum_{\theta\in\mathcal{P}_{\Omega}}f_{\theta} \right\|_{L^p(\R^{2d})}\leq \mathcal{D}^{\mathscr{C}}_{p,q}(\mathcal{P}_\Omega)\left(\sum_{\theta\in \mathcal{P}_{\Omega}}\|f_\theta\|_{L^p(\R^{2d})}^q\right)^\frac{1}{q},
\]
holds for any collection of functions with $\text{supp }\F_\sigma(f_{\theta})\subseteq \theta$.
\end{definition}
We refer to this as $\ell^qL^p$ decoupling. Decoupling has close ties to the Fourier restriction problem, as noted in \cite{Guth_23}. For instance in \cite{Laba_Wang_18}, \L aba and Wang used decoupling to prove existence of sets for which near optimal Fourier restriction is achieved.
In a recent paper, the Fourier restriction problem was extended to the setting of Schatten class operators on $L^2(\R^d)$ \cite{Luef_Samuelsen_24}. Here an equivalence between the classical and quantum Fourier restriction problem for compact sets was established using techniques from quantum harmonic analysis. A natural question is therefore if similar results can be achieved for decoupling.
To answer this question, we first introduce the operator version of the decoupling constant. We refer to the following as the quantum decoupling constant associated to $\ell^q\cS^p$ decoupling.
\begin{definition}
Let $\Omega\subset \R^{2d}$, and $\mathcal{P}_{\Omega}$ be a partition of $\Omega$.
We define the quantum decoupling constant $\mathcal{D}^{\mathscr{Q}}_{p,q}(\mathcal{P}_\Omega)$ to be the smallest constant for which the inequality
\[
\left\|\sum_{\theta\in\mathcal{P}_{\Omega}}T_{\theta} \right\|_{\cS^{p}}\leq \mathcal{D}^{\mathscr{Q}}_{p,q}(\mathcal{P}_\Omega)\left(\sum_{\theta\in \mathcal{P}_{\Omega}}\|T_\theta\|_{\cS^p}^q\right)^\frac{1}{q},
\]
holds for any collection of operators $\{T_\theta\}_{\theta\in\mathcal{P}_\Omega}\subset \mathcal{L}(\sS(\R^d);\sS'(\R^d))$ with $\mathrm{supp}\left(\F_W(T_{\theta})\right)\subseteq \theta$.
\end{definition}
Here $\F_W$ denotes the Fourier-Wigner transform. For a trace class operator $T$, the Fourier-Wigner transform is a function defined at a point $z\in\R^{2d}$ by
\[
\F_W(T)(z)=\tr(T\rho(-z)).
\]
Here $\rho$ denotes the symmetric time-frequency shift acting on $f\in L^2 (\R^d)$ by
\begin{equation}\label{Schr-rep}
\rho(x,\xi)f(t)=e^{-\pi i x\cdot \xi}e^{2\pi i \xi\cdot t}f(t-x).
\end{equation}

We are mainly interested in bounded linear operators on $L^2(\R^d)$. However, for technical reasons it is convenient to also consider continuous linear operators from $\sS(\R^d)$ to $\sS'(\R^d)$, denoted $\mathcal{L}(\sS(\R^{d});\sS'(\R^{d}))$. The main reason for considering these operators is that the Fourier-Wigner transform can be extended to an isomorphism between $\mathcal{L}(\sS(\R^{d});\sS'(\R^{d}))$ and $\sS'(\R^{2d})$. 
Since any bounded linear operator on $L^2(\R^d)$ is a continuous linear operator from $\sS(\R^d)$ into $L^2 (\R^d)\subset\sS'(\R^d)$, this allows us to define a Fourier-Wigner transform for general bounded linear operators on $L^2(\R^d)$.
The main result of the paper is the following.
\begin{theorem}\label{MainThm}
Let $\Omega\subset \R^{2d}$ be a bounded set, and let $\mathcal{P}_{\Omega}$ be a partition of $\Omega$. Then there exists $C=C(\Omega)\geq 1$ such that 
\[
C^{-1}\mathcal{D}_{p,q}^{\mathscr{C}}(\mathcal{P}_\Omega)\leq \mathcal{D}_{p,q}^{\mathscr{Q}}(\mathcal{P}_\Omega)
\leq C\,\mathcal{D}_{p,q}^{\mathscr{C}}(\mathcal{P}_\Omega)
\]
holds for any $1\leq p,q\leq \infty$.
\end{theorem}
Rather interesting is the fact that the constant in Theorem \ref{MainThm} relies only on $\Omega$ and not on the partition $\mathcal{P}_\Omega$. The fact that $\Omega$ is bounded is crucial. The proof uses a quantum version of Wiener's division lemma. More specifically, we use the fact that any function whose Fourier transform is compactly supported can be written as a convolution between itself and a suitable Schwartz function.

Since the constant $C(\Omega)$ is the same for all choices of $p$ and $q$, it is unlikely to be optimal. For instance, whenever $p=q=2$, the classical and quantum decoupling constants coincide by Plancherel's Theorem. Hence, we must have $C(\Omega)\equiv 1$ for the constant to be optimal in this case.

An immediate consequence can be seen by considering the truncated paraboloid 
\[
\mathbb{P}^{2d-1}:=\{(\xi,|\xi|^2)\in\R^{2d}: \xi\in \R^{2d-1}, |\xi_i|\leq1 \text{ for } 1\leq i\leq 2d-1\}.
\]
Given $0<\delta<1$, let $\mathcal{N}_{\mathbb{P}^{2d-1}}(\delta)$
denote the associated $\delta$-neighbourhood $\mathbb{P}^{2d-1}$ and
$\mathcal{P}_{\mathcal{N}_{\mathbb{P}^{2d-1}}(\delta)}$ a finitely overlapping cover of $\mathcal{N}_{\mathbb{P}^{2d-1}}(\delta)$ with
\[
\theta=\{(\xi,|\xi|^2+t): \xi\in C_\theta, |t|<\delta\},
\]
where $\{C_\theta\}$ is a partition of $[-1,1]^{2d-1}$ by cubes of side length $\delta^\frac{1}{2}$.
Combining Theorem \ref{MainThm} with the Bourgain-Demeter $\ell^2L^p$ decoupling for compact manifolds with non-vanishing Gaussian curvature \cite{Bourgain_Demeter_15} (see for instance Theorem $10.1$ in \cite{Demeter} for the case of the paraboloid), we end up with the following result.
 \begin{corollary}
For $0<\delta<1$ and any $\varepsilon>0$ there exists a constant $C_\varepsilon>0$ such that if $2\leq p\leq (4d+2)/(2d-1)$, then
\[
\mathcal{D}_{p,2}^{\mathscr{Q}}(\mathcal{P}_{\mathcal{N}_{\mathbb{P}^{2d-1}}(\delta)})\leq C_{\varepsilon} \delta^{-\varepsilon},
\]
while if $p\geq (4d+2)/(2d-1)$, then
\[
\mathcal{D}_{p,2}^{\mathscr{Q}}(\mathcal{P}_{\mathcal{N}_{\mathbb{P}^{2d-1}}(\delta)})\leq C_{\varepsilon} \delta^{-\varepsilon+\frac{2d+1}{2p}-\frac{2d-1}{4}}.
\]
\end{corollary}

\subsection{Structure of the paper}
The paper is divided into three further sections. We begin with a preliminary part which introduces necessary background material needed for the proof of Theorem \ref{MainThm}. This include selected topics from time-frequency analysis, operator theory and basic quantum harmonic analysis.

The proof of Theorem \ref{MainThm} utilises similar techniques as was done for the equivalence of the Fourier restriction problem in \cite{Luef_Samuelsen_24}. Werner's correspondence theory lets us transform the problem involving operators to a question about functions by considering operator-operator convolutions. By an operator version of Young's convolution inequality, we are then able to control the norm of the convolution from above. Unfortunately, unlike the restriction problem we also need to control the norm of the convolution from below. For this, we need a quantum version of Wiener's division lemma, which is provided in section \ref{DivisionLemmaSec}. Similar results have also been considered in the study of pseudodifferential operators with bandlimited Kohn-Nirenberg symbols \cite{Grochenig_Pauwels_14} and in the setting of Gabor multipliers in quantum harmonic analysis \cite{Skrettingland_20}. In \cite{dewage_Mitkovski_23}, a classical version of Wiener's division lemma was used to study Toeplitz operators. We end the paper by combining the techniques from the previous sections to give a complete proof of Theorem \ref{MainThm} in section $4$.

\section*{Acknowledgement}
The author would like to thank Franz Luef and Ben Johnsrude for insightful discussions in preparation of this manuscript, and Robert Fulsche for suggesting useful references.

\section{Preliminaries}
\subsection{Time-frequency analysis}
For a finite Borel measure $\mu$ on $\R^{2d}$, we define the symplectic Fourier transform at a point $\zeta\in\R^{2d}$ by
\[
\F_\sigma(\mu)(\zeta)=\int_{\R^{2d}}e^{-2\pi i \sigma(\zeta,z)}d\mu(z).
\]
Here $\sigma$ denotes the standard symplectic form on $\R^{2d}$ given by
\[
\sigma\left((x,\xi),(x',\xi')\right)=x'\cdot\xi-x\cdot\xi',\qquad (x,\xi),(x',\xi')\in\R^{2d}.
\]
The symplectic Fourier transform extends to a unitary operator on $L^{2}(\R^{2d})$ with the property that $\F_\sigma^2=\mathrm{Id}_{L^2}$.

If we let $\rho:\mathbb{R}^{2d}\to \mathcal{L}\left(L^2(\R^d);L^2(\R^d)\right)$ denote the symmetric time-frequency shift given by \eqref{Schr-rep}, then it gives rise to a time-frequency distribution known as the \textit{cross-ambiguity function}. For $f,g\in L^2(\R^d)$, the cross-ambiguity function is defined as
\begin{equation}\label{AmbiguityDef}
\mathcal{A}(f,g)(z):=\langle f,\rho(z)g\rangle_{L^2},\qquad z\in\R^{2d}.
\end{equation}
Moyal's identity ensures that $\mathcal{A}(f,g)\in L^2(\R^{2d})$ as long as $f,g\in L^2(\R^d)$ \cite[Chapter $3$]{Grochenig}. An important example is given by the cross-ambiguity function of the $L^2$-normalised Gaussian $\varphi_0(t)=2^{d/4}\exp(-\pi|t|^2)$ with itself. In this case, given $z\in \R^{2d}$, a direct computation yields
\begin{equation}\label{Ambi-Gauss}
\mathcal{A}\left(\varphi_0,\varphi_0\right)(z)=e^{-\pi\frac{|z|^2}{2}}.
\end{equation}
We also mention the covariance property of the cross-ambiguity function.
\begin{lemma}[Lemma $2.5$ in \cite{Luef_Samuelsen_24}]\label{Covariance property}
    Let $f,g\in L^2(\R^d)$. Then for any $z,\zeta\in \R^{2d}$
    \begin{equation*}
        \mathcal{A}(\rho(\zeta)f,g)(z)=e^{\pi i \sigma(\zeta,z)}\mathcal{A}(f,g)(z-\zeta).
    \end{equation*}
\end{lemma}

\subsection{Distributions}
We denote the Schwartz class by $\sS(\R^{2d})$, and its continuous dual space of tempered distributions by $\sS'(\R^{2d})$. For $\tau\in\sS'(\R^{2d})$ and $\psi\in \sS(\R^{2d})$, we use the sesquilinear dual pairing
\[
\langle \tau,\psi\rangle_{\sS',\sS}:=\tau(\overline{\psi}),
\]
in order to stay consistent with the $L^2$ inner product. Namely, for any $f\in L^2(\R^{2d})$,
\[
\langle f,\psi\rangle_{\sS',\sS}=\int_{\R^{2d}}f(x)\overline{\psi(x)}\,dx=\langle f,\psi\rangle_{L^2},
\]
for all $\psi\in \sS(\R^{2d})$. Using a sesquilinear dual pairing has the benefit of nicely extending the symplectic Fourier transform to an isomorphism on the tempered distributions through
\[
\langle \F_\sigma(\tau),\psi\rangle_{\sS',\sS}:=\langle \tau,\F_\sigma(\psi)\rangle_{\sS',\sS},
\]
for $\tau\in \sS'(\R^{2d})$ and all $\psi\in \sS(\R^{2d})$. Given $\eta\in \sS(\R^{2d})$ and $\tau\in \sS'(\R^{2d})$, we define the product $\eta\tau\in \sS'(\R^{2d})$ by
\[
\langle \eta\tau,\psi\rangle_{\sS',\sS}=\langle\tau,\overline{\eta}\psi\rangle_{\sS',\sS},
\]
for all $\psi\in \sS(\R^{2d})$. This is well-defined as the Schwartz class is closed under multiplication.

We denote the space of distributions with compact support by $\mathcal{E}'(\R^{2d})$, and this is a subspace of $\sS'(\R^{2d})$. It is well known that $\mathcal{E}'(\R^{2d})$ can be identified with the continuous dual space of $C^\infty(\R^{2d})$, and the Fourier transform of a compactly supported distribution is a smooth function.
\begin{lemma}[Theorem $7.1.14$ in \cite{Hormander-ALPDO1}]\label{Hormander-Lem}
Let $\tau\in \mathcal{E}'(\R^{2d})$. Then $\F_\sigma(\tau)$ is the $C^\infty$ function on $\R^{2d}$ given by
\[
\F_\sigma(\tau)(\zeta)=\langle \tau, e^{-2\pi i \sigma(\cdot,\zeta)}\rangle_{\sS',\sS}.
\]
\end{lemma}

We also mention the following lemma.
\begin{lemma}
Let $\tau\in \mathcal{E}'(\R^{2d})$, and let $\eta\in C_c^\infty(\R^{2d})$ be such that $\eta\equiv 1$ on $\mathrm{supp}(\tau)$. Then $\tau=\eta\tau$.
\end{lemma}

\subsection{Operator Theory}
We denote the space of bounded linear operators on $L^2(\R^d)$ by $\mathcal{L}(L^2(\R^d))$ and compact operators by $\mathcal{K}$. For each $T\in \mathcal{K}$ there exists a sequence of non-negative numbers $\{s_n(T)\}_{n\in\mathbb{N}}\subset [0,\infty)$ with $s_n(T)\to 0$ as $n\to\infty$, and two orthonormal families $\{e_n\}_{n\in\mathbb{N}}$ and $\{\eta_n\}_{n\in\mathbb{N}}$ in $L^2(\R^d)$ such that
\[
T=\sum_{n\in\mathbb{N}}s_n(T)e_n\otimes \eta_n,
\]
where $e_n\otimes \eta_n (f)=\langle f,\eta_n\rangle e_n$ is a rank-one operator. This is known as the singular value decomposition of $T$, and $\{s_n(T)\}_{n\in\mathbb{N}}$ are the singular values of $T$.

For $1\leq p\leq \infty$, define the Schatten $p$-class by 
\[
\cS^p:=\{T\in\mathcal{L}(L^2(\R^d)):\{s_n(T)\}_{n\in\mathbb{N}}\in \ell^p\},
\]
where we make the identification $\cS^\infty=\mathcal{L}(L^2(\R^d))$. By the properties of $\ell^p$-spaces, it follows that $\cS^1\subseteq \cS^p\subseteq \cS^q\subseteq \mathcal{K}$ for $q>p$. The Schatten $p$-class is a Banach space when equipped with the norm
\[
\|T\|_{\cS^p}:=\left(\sum_{n\in\mathbb{N}}|s_n(T)|^p\right)^\frac{1}{p},
\]
and for $1<p<\infty$ the dual space is identified with $\cS^{p'}$ where $p'$ is the H\"{o}lder conjugate of $p$. 
If $T=\upsilon\otimes \eta$ for some $\upsilon, \eta\in L^2(\R^d)$, then $\|T\|_{\cS^p}=\|\upsilon\|_{L^2(\R^d)}\|\eta\|_{L^2(\R^d)}$ as this is the only non-zero singular value of the rank-one operator.
For more details on Schatten classes, we refer to chapter $3$ of \cite{SimonOp}.

\subsection{Fourier Analysis of Operators}
Let $\rho:\R^{2d}\to \mathcal{L}(L^2(\R^d))$ denote the symmetric time-frequency shift defined by \eqref{Schr-rep}. For $T\in\cS^1$ we define the Fourier-Wigner transform at a point $z\in\R^{2d}$ by
\[
\F_W(T)(z)=\tr(T\rho(-z))=\langle T,\rho(z)\rangle_{\cS^1,\cS^\infty}.
\]
The function $\F_W(T)$ is known as the spreading function of the operator $T$ in time-frequency analysis \cite{Feichtinger_Kozak_98}. The spreading function allows us to write a pseudodifferential operator as a superposition of time-frequency shifts \cite[Chapter $14$]{Grochenig}.
If $F$ denotes the spreading function of a bounded linear operator $T$, then we can formally recover the operator by the inverse Fourier-Wigner transform,
\[
T=\F_W^{-1}(F)=\int_{\R^{2d}}F(z)\rho(z)\,dz.
\]
This is a continuous linear map from $L^1(\R^{2d})$ into $\mathcal{K}$ by Theorem $1.30$ in \cite{Folland_Phase}. 

Many basic properties of $\F_W(T)$ are investigated in \cite{Werner}. A version of the Riemann-Lebesgue lemma holds. If $T\in \cS^1$, then $\F_W(T)$ belongs to the space of continuous functions vanishing at infinity, denoted $C_0(\R^{2d})$. The Fourier-Wigner transform even extends to a unitary operator from $\cS^2$ to $L^2(\R^{2d})$ as a consequence of a result by Pool \cite{Pool}. Through interpolation we regain Hausdorff-Young's inequality.
\begin{prop}[Proposition $3.4$ in \cite{Werner}]
Let $1\leq p\leq 2$ and assume  $T\in\cS^p$. Then $\F_W(T)\in L^{p'}(\R^{2d})$, and
\[
\|\F_W(T)\|_{L^{p'}(\R^{2d})}\leq \|T\|_{\cS^p}.
\]
\end{prop}

Consider the rank-one operator $g\otimes h$ with $g,h\in L^2(\R^d)$. The Fourier-Wigner transform of $g\otimes h$ turns out to be the cross-ambiguity function given by \eqref{AmbiguityDef},
\[
\F_W(g\otimes h)(z)=\mathcal{A}(g,h)(z).
\]

Let $Pf(x)=f(-x)$ denote the parity operator, while $\alpha_z(T)=\rho(z)T\rho(-z)$ is conjugating an operator with the symmetric time-frequency shift. Let $S,T\in\cS^1$ and $f\in L^1(\R^{2d})$. We then define an operator-operator convolution by
\[
T\star S(z):=\tr(T\alpha_z(PSP)),
\]
and a function-operator convolution by
\[
f\star T=T\star f:=\int_{\R^{2d}}f(z)\alpha_z(T)\,dz,
\]
where the integral is considered weakly.
Note that $T\star S$ gives a function, while $T\star f$ defines an operator. Associated to these convolutions is a version of Young's convolution inequality which we call Werner-Young's inequality. This was first proven in \cite{Werner}, but we state it as given in \cite{Luef_Skrettingland_19}.
\begin{theorem}[Werner-Young's inequality]\label{Werner-Young}
Let $1\leq p,q,r\leq \infty$ be such that $1+r^{-1}=p^{-1}+q^{-1}$. If $f\in L^p(\R^{2d})$, $T\in \cS^p$ and $S\in \cS^q$, then $f\star S\in \cS^r$ and $T\star S\in L^{r}(\R^{2d})$. Moreover, there are the norm bounds
\begin{align*}
\|f\star S\|_{\cS^r}\leq& \|f\|_{L^p}\|S\|_{\cS^q},\\
\|T\star S\|_{L^r}\leq& \|T\|_{\cS^p}\|S\|_{\cS^q}.
\end{align*}
\end{theorem}

These convolutions interact with the Fourier-Wigner transform in the same way as the classical convolution interacts with the Fourier transform.
\begin{theorem}[Proposition $3.12$ in \cite{Luef_Skrettingland_19}]
Let $f\in L^1$, $S,T\in \cS^1$. Then the following holds,
\begin{align*}
\F_W(f\star S)=\F_\sigma(f)\F_W(S),\\
\F_\sigma(T\star S)=\F_W(T)\F_W(S).
\end{align*}
\end{theorem}

\subsection{Weyl calculus, Schwartz operators and tempered distributions}
An application of Schwartz' kernel theorem gives that for any continuous linear operator $T:\mathscr{S}(\R^d)\to \sS'(\R^d)$ we can associate a tempered distribution $a_T\in\mathscr{S}'(\R^{2d})$, called the Weyl symbol of $T$, which is defined through the relation
\begin{equation}\label{WeylQuant}
\langle T \psi,\varphi\rangle_{\sS',\sS}=\langle a_T, \mathcal{W}(\varphi,\psi)\rangle_{\sS',\sS},
\end{equation}
for all $\psi,\varphi\in\mathscr{S}(\R^d)$. Here $\mathcal{W}$ denotes the cross-Wigner distribution defined by
\[
\mathcal{W}(\varphi,\psi)(x,\xi):=\int_{\R^d}\varphi\left(x+\frac{t}{2}\right)\overline{\psi\left(x-\frac{t}{2}\right)}e^{-2\pi i t\cdot\xi}\,dt=\F_\sigma\left(\mathcal{A}(\varphi,\psi)\right)(x,\xi).
\]
Since any $a\in\sS'(\R^{2d})$ necessarily defines a continuous linear map from $\sS(\R^{d})$ to $\sS'(\R^d)$ through \eqref{WeylQuant}, we see that there is an isomorphism between $\sS'(\R^{2d})$ and $\mathcal{L}(\sS(\R^d);\sS'(\R^d))$. For a more detailed description we refer to Theorem $14.3.5$ in \cite{Grochenig}.

\begin{definition}
Let $\mathfrak{S}$ denote the space of all continuous operators $L_a:\mathscr{S}(\R^{d})\to \mathscr{S}'(\R^{d})$ for which the Weyl symbol is in $\mathscr{S}(\R^{2d})$, i.e.
\begin{equation*}
\mathfrak{S}:=\left\{L_a:\mathscr{S}(\R^d)\to\mathscr{S}'(\R^d): a\in \mathscr{S}(\R^{2d})\right\}.
\end{equation*}
\end{definition}
The class $\fS$ is a subspace of $\cS^1$ by Proposition $4.1$ in \cite{Grochenig-Heil}. For a detailed study of $\mathfrak{S}$, we refer to \cite{Keyl-Kiukas-Werner_16}.
We can define an isomorphism $\iota:\mathscr{S}(\R^{2d})\to \mathfrak{S}$ by $\iota(\varphi)=L_\varphi$, which induces an isomorphism $\iota^*:\mathfrak{S}'\to \mathscr{S}'(\R^{2d})$ given by
\[
\langle B,\iota (\varphi)\rangle_{\mathfrak{S}',\mathfrak{S}}=\langle \iota^*(B),\varphi\rangle_{\mathscr{S}',\mathscr{S}}.
\]
By Schwartz' kernel theorem we may identify $B$ with the continuous operator $L_{\tau^*(B)}:\mathscr{S}(\R^{d})\to\mathscr{S}'(\R^d)$, and thus we can extend the Fourier-Wigner transform to an isomorphism $\fS'\to\sS'(\R^{2d})$ through the identification
\[
\F_W(B)=\F_\sigma({\iota^*B}).
\]
This also allows us to define the Fourier-Wigner transform of a bounded linear operator on $L^2(\R^d)$ by first restricting it to an operator acting on $\sS(\R^d)$, and then identifying it with an element of $\fS'$. From here on, we will identify $\fS'$ with $\mathcal{L}(\sS(\R^{2d});\sS'(\R^{2d}))$ and denote $\iota^*T=a_T$, which we call the Weyl symbol of $T\in\fS'$.

\begin{definition}
For $T\in\mathfrak{S}'$ we define $\F_W(T)$ to be the tempered distribution such that
\[
\langle T,S\rangle_{\mathfrak{S}',\mathfrak{S}}=\langle\F_W(T),\F_W(S)\rangle_{\mathscr{S}',\mathscr{S}},
\]
holds for all $S\in \mathfrak{S}$.
\end{definition}
\begin{remark}
For $T\in\fS'$, let $a_T=\iota^* T$ denote the Weyl symbol of $T$. Then for any $\psi\in\sS(\R^{2d})$,
\[
\langle T,L_\psi\rangle_{\fS',\fS}=\langle a_T,\psi\rangle_{\sS',\sS}=\langle \F_\sigma(a_T),\F_\sigma(\psi)\rangle_{\sS',\sS}=\langle \F_{\sigma}(a_T),\F_W(L_\psi)\rangle_{\sS',\sS},
\]
where $\F_\sigma^{-1}=\F_\sigma$ was used for the second equality. This shows that $\F_W(T)=\F_\sigma(a_T)$.
\end{remark}
The following result was originally proved in section $5$ of \cite{Keyl-Kiukas-Werner_16} as several different results, and summarised in \cite{Luef_Skrettingland_19}.
\begin{prop}[Proposition $3.16$ in \cite{Luef_Skrettingland_19}]\label{ConvProp-Distribution}
Let $S,T\in\fS$, $\psi\in\sS(\R^{2d})$, $\tau\in\sS'(\R^{2d})$ and $A\in \fS'$. Then:
\begin{enumerate}[i)]
\item The following convolution relations hold,
\begin{align*}S\star T\in \sS(\R^{2d}),&\qquad \psi\star S\in\fS,\\
S\star A\in \sS'(\R^{2d}),&\qquad \psi\star A\in \fS',\\
\psi*\tau\in\sS'(\R^{2d}),&\qquad S\star \tau\in \fS'.
\end{align*}
\item $\F_W$ extends to a topological isomorphism from $\fS'$ to $\sS'(\R^{2d})$.
\item The relations \begin{align*}
\F_\sigma(S\star A)=&\,\F_W(S)\F_W(A),\, &&\F_W(\psi\star A)=\F_\sigma(\psi)\F_W(A),\\
\F_\sigma(\psi*\tau)=&\,\F_\sigma(\psi)\F_\sigma(\tau),\, &&\F_W(S\star \tau)=\F_W(S)\F_\sigma(\tau),
\end{align*}
hold.
\item The Weyl symbol of $A$ is given by $a_A=\F_\sigma(\F_W(A))$.
\end{enumerate}
\end{prop}
\section{A convolution result for operators and functions on phase space}\label{DivisionLemmaSec}
In this section we introduce the main component of the proof of Theorem \ref{MainThm}. By convolving with rank-one operators we are able to transfer between operators and functions on $\R^{2d}$. Werner-Young's inequality (Theorem \ref{Werner-Young}) then gives that the map $S\mapsto S\star(g\otimes h)$ maps $\cS^p$ to $L^p(\R^{2d})$ for fixed $g,h\in L^2(\R^d)$. Moreover, we also control the $L^p$ norm of the convolution as
$\|S\star(g\otimes h)\|_{L^p}\leq \|S\|_{\cS^p}$ whenever $g,h$ are $L^2$ normalised.


In what follows, we are interested in controlling the $L^p$ norm of the convolution $S\star(g\otimes h)$ by the Schatten $p$ norm of $S$ from below. This would naturally follow from Theorem \ref{Werner-Young} if we could write the operator as a triple convolution, say $S=S\star(g\otimes h)\star B$ for some $B\in \cS^1$. If this is the case, then by applying Theorem \ref{Werner-Young} a second time, we would end up with
\[
\|S\|_{\cS^p}\leq \|B\|_{\cS^1}\|S\star(g\otimes h)\|_{L^p(\R^{2d})}.
\]

The next result shows that any $S\in\fS'$ can be written as $S=S\star(g\otimes h)\star B$ for some operator $B$ under the assumption that $\F_W(S)$ is compactly supported. This result is inspired by a classical argument about functions whose Fourier transforms are compactly supported \cite[Chapter $5$]{Wolff}. If $f\in \sS(\R^{2d})$ and $\F_\sigma(f)$ is supported on the unit ball $B_1(0)$, then given a smooth cut-off function $\psi\in C_c^\infty(\R^{2d})$ with $\psi\equiv 1$ on $B_1(0)$, it follows that
\[
\F_\sigma(f)(\xi)=\F_\sigma(f)(\xi)\psi(\xi)=\F_\sigma(f*\F_\sigma(\psi))(\xi).
\]
This implies that $f=f*\F_\sigma(\psi)$. With this in mind, we want to extend the same argument to distributions and operators. We therefore present the following proposition.
\begin{prop}\label{Division_prop}
Let $\Omega$ be a bounded set in $\R^{2d}$. There exist $L^2$-normalised $g,h\in \sS(\R^{d})$ and $B\in \fS$ such that:
\begin{enumerate}[i)]
\item If $\tau\in\sS'(\R^{2d})$ and $\F_\sigma(\tau)$ is supported on $\Omega$, then $\tau=\tau\star(g\otimes h)\star B$.
\item If $T\in\fS'$ and $\F_W(T)$ is supported on $\Omega$, then $T=T\star(g\otimes h)\star B$.
\end{enumerate}
\end{prop}

\begin{proof}
Since $\Omega$ is bounded, we can fix some $c_\Omega\in \Omega$ such that $\Omega\subseteq B_R(c_\Omega)$ for some $R>0$. Choose two $L^2$-normalised functions $g,h\in\mathscr{S}(\R^{d})$ such that $|\mathcal{A}(g,h)|\geq \delta>0$ on $B_{2R}(c_\Omega)$ for some $\delta>0$. Let $\Psi\in C_c^\infty(\R^{2d})$ be supported on $B_{2R}(c_\Omega)$ and $\Psi|_\Omega\equiv 1$.
Note that we can write
\[
\Psi(z)=\frac{\mathcal{A}(g,h)(z)}{\mathcal{A}(g,h)(z)}\Psi(z),
\]
as $|\mathcal{A}(g,h)(z)|\geq \delta>0$ for any $z$ on the support of $\Psi$. Moreover, it follows that for each $z\in \R^{2d}$,
\[
\Psi(z)=\F_W(g\otimes h)(z)\F_W(B)(z)
\]
where $B$ is defined as
\[
B=\F_W^{-1}\left(\frac{\Psi}{\mathcal{A}(g,h)}\right)=\int_{\R^{2d}}\frac{1}{\mathcal{A}(g,h)(z)}\Psi(z)\rho(z)\,dz.
\]
We know that $1/\mathcal{A}(g,h)$ is a bounded $C^\infty$ function as $\mathcal{A}(g,h)$ is bounded away from zero on $B_{2R}(c_\Omega)$ and since the cross-ambiguity function of two Schwartz functions is in $\sS(\R^{2d})$ by Theorem $11.2.5$ in \cite{Grochenig}.
Thus, the spreading function of $B$ is a $C_c^\infty$-function, and it follows that $B\in \fS$ since the Weyl symbol of an operator is the symplectic Fourier transform of the spreading function by Proposition \ref{ConvProp-Distribution}.

Whenever $\F_\sigma(f)\in \sS'(\R^{2d})$ is supported on $\Omega$, it follows for any $\eta\in \sS(\R^{2d})$ that
\begin{align*}
\langle \F_\sigma(f),\eta\rangle_{\sS',\sS}=\langle \Psi\F_\sigma(f),\eta\rangle_{\sS',\sS}
=&\langle\F_\sigma(f)\F_W(g\otimes h)\F_W(B),\eta\rangle_{\sS',\sS}\\
=&\langle\F_W(f\star (g\otimes h))\F_W(B),\eta\rangle_{\sS',\sS}\\
=&\langle\F_\sigma(f\star(g\otimes h)\star B),\eta\rangle_{\sS',\sS}.
\end{align*}
This shows that $\F_\sigma(f)=\F_\sigma(f\star(g\otimes h)\star B)
$ in $\sS'$. Since the symplectic Fourier transform is an isomorphism on $\sS'$, we can conclude that
\[
f=f\star(g\otimes h)\star B.
\]
This proves assertion $i)$.

For $ii)$ we use a similar argument. Since $\F_W(T)$ is supported on $\Omega$ it follows that for any $\eta\in \sS(\R^{2d})$,
\begin{align*}
\langle \F_W(T),\eta\rangle_{\sS',\sS}=&\langle\Psi\F_W(T),\eta\rangle_{\sS',\sS}\\
=&\langle \F_W(T)\F_W(g\otimes h)\F_W(B),\eta\rangle_{\sS',\sS}\\
=&\langle\F_W\left(T\star \left(g\otimes h\right)\star B\right),\eta\rangle_{\sS',\sS}.
\end{align*}
We therefore conclude that $T=T\star(g\otimes h)\star B$ as the Fourier-Wigner transform is an isomorphism from  $\fS'$ to $\sS'(\R^{2d})$.
\end{proof}
\begin{remark}
One possible explicit choice of $g$ and $h$ can for instance be
\[
g(y)=\varphi_{c_\Omega}(y)=\rho(c_{\Omega})\varphi_0(y)=2^{\frac{d}{4}}e^{-\pi i x_\Omega\cdot \xi_\Omega}e^{2\pi i \xi_\Omega\cdot y}e^{-\pi|y-x_\Omega|^2},\quad h(y)=\varphi_0(y)=2^{\frac{d}{4}}e^{-\pi |y|^2}.
\]
where $c_\Omega=(x_\Omega,\xi_\Omega)$ is chosen as in the proof. For this choice we have by Lemma \ref{Covariance property}
\[
\F_W(g\otimes h)(z)=\mathcal{A}(g,h)(z)=e^{\pi i \sigma(c_\Omega,z)}e^{-\pi\frac{|z-c_\Omega|^2}{2}},
\]
and thus
\[
|\mathcal{A}(g,h)(z)|\geq e^{-2\pi R^2}>0,
\]
for all $z\in B_{2R}(c_\Omega)$.
The operator $B$ is then given by
\[
B=\int_{\R^{2d}}\Psi(z)e^{-\pi i \sigma(c_\Omega,z)+\pi \frac{|z-c_\Omega|^2}{2}}\rho(z)\,dz.
\]
\end{remark}

The Paley-Wiener-Schwartz lemma ensures that any tempered distribution $\tau\in\sS'(\R^{2d})$ for which the Fourier transform is compactly supported is in fact a $C^\infty$ function of at most polynomial growth on $\R^{2d}$ \cite[Chapter $7$]{Hormander-ALPDO1}. This means that we can define an $L^p$ norm of such distributions, even though the integral might not converge. This leads to the following corollary.

\begin{corollary}\label{LW-bnd-cor}
Let $\Omega$ be a bounded set, and $1\leq p\leq \infty$. There exists $L^2$-normalised $g,h\in \sS(\R^{d})$ and $C=C(\Omega)>0$ such that 
\begin{enumerate}[i)]
\item If $\tau\in \sS'(\R^{2d})$ and $\F_\sigma(\tau)$ is supported on $\Omega$, then
\[
\|\tau\|_{L^p(\R^{2d})}\leq C(\Omega)\|\tau\star(g\otimes h)\|_{\cS^p}.
\]
\item If $T\in\fS'$ and $\F_W(T)$ is supported on $\Omega$, then
\[
\|T\|_{\cS^p}\leq C(\Omega) \|T\star(g\otimes h)\|_{L^p(\R^{2d})}.
\]
\end{enumerate}
\end{corollary}
\begin{proof}
This is a direct consequence of Proposition \ref{Division_prop} and Theorem \ref{Werner-Young}. Namely, for $i)$ we have
\[
\|\tau\|_{L^p(\R^{2d})}=\|\tau\star\left(g\otimes h\right)\star B\|_{L^p(\R^{2d})}\leq \|B\|_{\cS^1}\|T\star \left(g\otimes h\right)\|_{\cS^p},
\]
while for $ii)$
\[
\|T\|_{\cS^p}=\|T\star\left(g\otimes h\right)\star B\|_{\cS^p}\leq \|B\|_{\cS^1}\|T\star \left(g\otimes h\right)\|_{L^p(\R^{2d})}.
\]
Since the operator 
\[
B=\int_{\R^{2d}}\frac{1}{\mathcal{A}(g,h)(z)}\Psi(z)\rho(z)\,dz
\]
only depends on the set $\Omega$, we see that 
the constant 
\[
C=\|B\|_{\cS^1}<\infty,
\]
depends only on $\Omega$ as well.
\end{proof}
\begin{remark}
If $\tau\in \sS'(\R^{2d})$ and $\F_\sigma(\tau)$ is supported on a bounded set $\Omega$, then a second application of Werner-Young's inequality results in 
\[
\|\tau\|_{L^p(\R^{2d})}\leq C(\Omega)\|\tau\star (g\otimes h)\|_{\cS^p}\leq C(\Omega)\|g\|_{L^2(\R^d)}\|h\|_{L^2(\R^d)}\|\tau\|_{L^p(\R^{2d})}.
\]
This shows that the norms are equivalent, and we have convergence if and only if $\tau\in L^p(\R^{2d})$. The same argument also holds for $T\in \fS'$ with $\F_W(T)$ supported on $\Omega$.

\end{remark}
\begin{remark}
There is the trivial lower bound on the constant $C(\Omega)$,
\[
\|B\|_{\cS^1}\geq \|\F_{W}(B)\|_{L^\infty(\R^{2d})}\geq \frac{1}{|\mathcal{A}(g,h)(c_\Omega)|}|\Psi(c_\Omega)|\geq\frac{1}{\|g\|_{L^2}\|h\|_{L^2}}=1,
\]
as both $g$ and $h$ are $L^2$-normalised.
\end{remark}

\section{Proof of Theorem \ref{MainThm}}
In this section we present the proof of Theorem \ref{MainThm}. The proof is divided into two propositions, each proving one of the inequalities in the theorem. 
\begin{prop}
Let $\Omega\subset \R^{2d}$ be a bounded set and $\mathcal{P}_\Omega$ a partition of $\Omega$. Then there exists $C(\Omega)>0$ such that
\[
\mathcal{D}_{p,q}^\mathscr{Q}(\mathcal{P}_\Omega)\leq C(\Omega)\,\mathcal{D}_{p,q}^\mathscr{C}(\mathcal{P}_\Omega).
\]
\end{prop}
\begin{proof}
Let $T=\sum_{\theta\in \mathcal{P}_\Omega}T_\theta$ with $\mathrm{supp}(\F_W(T_\theta))\subseteq \theta$. By part $ii)$ in corollary \ref{LW-bnd-cor} it follows that for some $L^2$ normalised $g,h\in\mathscr{S}(\R^{d})$,
\begin{align*}
\left\|\sum_{\theta\in \mathcal{P}_\Omega}T_\theta\right\|_{\cS^p}
\leq C(\Omega)\left\|\left(\sum_{\theta\in \mathcal{P}_\Omega}T_\theta\right)\star ( g\otimes h)\right\|_{L^p(\R^{2d})}
=C(\Omega)\left\|\sum_{\theta\in \mathcal{P}_\Omega}T_\theta\star ( g\otimes h)\right\|_{L^p(\R^{2d})}.
\end{align*}
Define $F_\theta:=T_\theta\star( g\otimes h)$, which is a $C^\infty(\R^{2d})$ function by Lemma \ref{Hormander-Lem}. By Proposition \ref{ConvProp-Distribution} it follows that
\[
\F_\sigma(F_\theta)=\F_W( g\otimes h)\F_W(T_\theta),
\]
and so $\mathrm{supp}(\F_\sigma(F_\theta))\subseteq \mathrm{supp}(\F_W(T_\theta))$. By the definition of the classical decoupling constant,
\begin{align*}
\left\|\sum_{\theta\in \mathcal{P}_\Omega}T_\theta\right\|_{\cS^{p}}\leq C(\Omega)\left\|\sum_{\theta\in \mathcal{P}_\Omega}F_\theta\right\|_{L^{p}(\R^{2d})}\leq C(\Omega)\mathcal{D}_{p,q}^{\mathscr{C}}(\mathcal{P}_{\Omega})\left(\sum_{\theta\in\mathcal{P}_\Omega}\|F_\theta\|^q_{L^{p}(\R^{2d})}\right)^\frac{1}{q}.
\end{align*}
To finish the proof, we invoke Theorem \ref{Werner-Young} to conclude that
\[
\|F_\theta\|_{L^p(\R^{2d})}=\|T_\theta\star( g\otimes h)\|_{L^{p}(\R^{2d})}\leq \| g\otimes h\|_{\cS^1}\|T_\theta\|_{\cS^p}=\| h\|_{L^2(\R^d)}\|g\|_{L^2(\R^d)} \|T_\theta\|_{\cS^p}.
\]
However, both $g$ and $h$ are $L^2$ normalised and so we can conclude that
\[
\left\|\sum_{\theta\in \mathcal{P}_\Omega}T_\theta\right\|_{\cS^{p}}\leq C(\Omega)\mathcal{D}_{p,q}^{\mathscr{C}}(\mathcal{P}_{\Omega})\left(\sum_{\theta\in\mathcal{P}_\Omega}\|T_\theta\|^q_{\cS^p}\right)^\frac{1}{q}.
\]
Since $\mathcal{D}_{p,q}^\mathscr{Q}(\mathcal{P}_\Omega)$ is the smallest constant for which the inequality holds, it follows that
\[
\mathcal{D}_{p,q}^\mathscr{Q}(\mathcal{P}_\Omega)\leq C(\Omega)\mathcal{D}_{p,q}^{\mathscr{C}}(\mathcal{P}_{\Omega}).
\]
\end{proof}

For the other direction, we have the following proposition which is proved in a similar manner by again utilising Corollary \ref{LW-bnd-cor}. 
\begin{prop}
Let $\Omega\subset \R^{2d}$ be a bounded set and $\mathcal{P}_\Omega$ a partition of $\Omega$. Then there exists $C(\Omega)>0$ such that
\[
\mathcal{D}_{p,q}^\mathscr{C}(\mathcal{P}_\Omega)\leq C(\Omega)\,\mathcal{D}_{p,q}^\mathscr{Q}(\mathcal{P}_\Omega).
\]
\end{prop}
\begin{proof}
Let $\{f_\theta\}_{\theta\in\mathcal{P}_\Omega}$ be a collection of functions such that $\mathrm{supp}\left(\F_\sigma(f_\theta)\right)\subseteq \theta$ for each $\theta$. Define the function 
\[
f=\sum_{\theta\in \mathcal{P}_\Omega}f_\theta,
\]
for which we have $\mathrm{supp}(\F_\sigma(f))\subseteq \Omega$.
By part $i)$ in Corollary \ref{LW-bnd-cor} there exist $L^2$ normalised $g,h\in\sS(\R^d)$ and a constant $C=C(\Omega)>0$ such that
\begin{equation}\label{Func-to-Op-eq-1}
\left\|\sum_{\theta\in \mathcal{P}_\Omega}f_\theta\right\|_{L^p(\R^{2d})}\leq C\, \|f\star(  g\otimes h)\|_{\cS^p}
=C\, \left\|\sum_{\theta\in\mathcal{P}_\Omega}f_\theta\star( g\otimes h)\right\|_{\cS^p}.
\end{equation}
Note however that $\mathrm{supp}(\F_W(f_\theta\star( h\otimes\varphi)))\subseteq \theta$ as $\F_\sigma(f_\theta)$ is supported on $\theta$ and
\[
\F_W(f_\theta\star( h\otimes\varphi))=\F_\sigma(f_\theta)\F_W(  g\otimes h),
\]
by Proposition \ref{ConvProp-Distribution}.
Thus, combining \eqref{Func-to-Op-eq-1} with the definition of the quantum decoupling constant, we see that
\[
\left\|\sum_{\theta\in\mathcal{P}_\Omega}f_\theta\right\|_{L^{p}(\R^{2d})}\leq C\, \left\|\sum_{\theta\in\mathcal{P}_\Omega}f_\theta\star( g\otimes h)\right\|_{\cS^p}\leq C\mathcal{D}_{p,q}^{\mathscr{Q}}(\mathcal{P}_\Omega)\left(\sum_{\theta\in\mathcal{P}_\Omega}\|f_\theta\star(  g\otimes h)\|_{\cS^p}^q\right)^\frac{1}{q}.
\]
By Theorem \ref{Werner-Young}, it follows that $\|f_\theta\star(  g\otimes h)\|_{\cS^p}\leq \|f_\theta\|_{L^p}$ for each $\theta$ as both $g$ and $h$ are $L^2$ normalised. Thus, we conclude that
\[
\left\|\sum_{\theta\in\mathcal{P}_\Omega}f_\theta\right\|_{L^{p}(\R^{2d})}\leq C\mathcal{D}_{p,q}^{\mathscr{Q}}(\mathcal{P}_\Omega)\left(\sum_{\theta\in\mathcal{P}_\Omega}\|f_\theta\|_{L^p(\R^{2d})}^q\right)^\frac{1}{q},
\]
which shows that
\[
\mathcal{D}_{p,q}^{\mathscr{C}}(\mathcal{P}_\Omega)\leq C(\Omega)\mathcal{D}_{p,q}^{\mathscr{Q}}(\mathcal{P}_\Omega).
\]
\end{proof}

\printbibliography

\end{document}